\documentclass[12pt]{amsart}

\usepackage {amsmath}
\usepackage[applemac]{inputenc}
\usepackage {amsfonts, amssymb}
\usepackage {latexsym}

\newtheorem {theorem} {Theorem} [section]
\newtheorem {lemma} {Lemma} [section]

\newtheorem{preremark}{Remark}[section]
\newenvironment{remark}%
  {\begin{preremark}\rm}{\end{preremark}}
   \newtheorem{preremark1}{Example}[section]

\newcommand {\R} {\mathbb{R}}
\def\be{\begin{equation}}
\def\ee{\end{equation}}

\begin{document}

\title[Flashing ratchet]{The flashing ratchet and unidirectional transport of matter}

\author[D.Vorotnikov]{Dmitry Vorotnikov}

\address{CMUC, Department of
Mathematics, University of Coimbra, 3001-454 Coimbra, Portugal}{}
\email{mitvorot@mat.uc.pt}

\thanks{The author would like to thank David Kinderlehrer for drawing his attention to the problem and Michal Kowalczyk and Ivan Yudin for discussions. The research was partially supported by CMUC/FCT and CMU-Portugal Program.}

\keywords{flashing ratchet, Brownian motor, Fokker-Planck equation, periodic solution, Markov chain, transport}
\subjclass[2010]{35B10; 35Q84; 60J10; 60J70; 82C70}



\begin{abstract}
We study the flashing ratchet model of a Brownian motor, which consists in cyclical switching between the Fokker-Planck equation with an asymmetric ratchet-like potential and the pure diffusion equation. We show that the motor really performs unidirectional transport of mass, for proper parameters of the model, by analyzing the attractor of the problem and the stationary vector of a related Markov chain.
\end{abstract}

\maketitle

\section {Introduction}

Nano-scale or molecular devices which use energy but not momentum to generate transport are called \emph{Brownian motors}. Such phenomena arise in various areas of science, from intracellular transport to nanotechnology \cite{ast,ast1,bras,pet,bm}.

The general relation for various types of fluctuation-driven motors looks like \cite{ast}
\be \label{eqi} \rho_t= \sigma\rho_{xx}+ (\Psi_x\rho)_x,\quad
 x\in (0,1); t>0. \ee

Here $\rho$ is the unknown density, $\sigma$ is the diffusion coefficient, and $\Psi(x,t)$ is the potential. For the \emph{flashing ratchet}, 
an autonomous potential $\psi$ is switched on and off cyclically \cite{ast}, i.e. \cite{kind1,kind2} \be\label{pot}\Psi(x,t)=h(t)\psi(x), \ee where
\be \label{h} h(t)=\left\{\begin{array}{ll} 1,
& nT<t\leq nT+T_{tr},\ n=0,1,\dots,\\
0, & nT+T_{tr}<t\leq nT+T,\ n=0,1,\dots
\end{array}\right.\ee
A typical ratchet-like potential $\psi$ with $k$ teeth, $k>1$, is $1/k$-periodic in $x$ and has a unique local (and, hence, global) minimum within each period.

In \cite{kind1} it was shown that the behaviour of the flashing ratchet system (with Neumann boundary conditions) is in some sense close to the behaviour of a certain Markov chain. It was observed that having this at hand it is possible to verify transport via comparing eventual distribution of mass between the ''wells'' of the potential $\psi$, i.e. the line segments with end points at successive maxima of $\psi(x)$. Any inequality in this distribution would mean transport. In particular, it was shown that exactly this takes place for proper parameters of the ratchet and $k=2$. However, the proof was not completely consistent, being based on incorrect time asymptotics of the second derivative of the Green function (a power function instead of an exponential one). Generalization of this claim to the case $k>2$ was mentioned as an open problem in \cite{kind2,kind1}. 

In this paper, we give evidence of unidirectional transport for any $k>1$.  Namely, we show (Theorem \ref{mainthm}) that for certain parameters of the flashing ratchet, after a sufficiently large number of cycles the amount of mass in the wells of the potential is strongly decreasing/increasing from the left to the right, provided the minima of $\psi(x)$ are located in the left/right halves of the wells. 

It is important to note that the transport provided by the flashing ratchet is due to flashing \eqref{h} only, since both pure diffusion ($h\equiv 0$) and ''perpetual ratchet'' ($h\equiv 1$) with a  periodic potential rapidly approach their equilibria without any specific right or left drift tendency. 

Let us also recall that there is a connection (see e.g. \cite{ameng,heath,kind1}) between the flashing ratchet, especially the fact that it produces unidirectional transport, and Parrondo's paradox in game theory, where a well-scheduled alternation in playing two fair or losing games becomes a winning strategy.  

The paper is organized as follows. In the next section, we present the problem more rigorously, give necessary notations and facts, and formulate the main result (Theorem \ref{mainthm}). In the third section, we demonstrate that the so-called discrete ratchet, which generates the Markov chain, behaves in a way similar to the claimed behaviour of the flashing ratchet. The proof of the main result is provided in the final section.

\section {Preliminaries}

We consider the boundary value problem for the flashing ratchet equation with Neumann boundary conditions:

\be \label{eq}
\left\{\begin{array}{ll} \rho_t= \sigma\rho_{xx}+ h(t)(b\rho)_x,
& x\in (0,1); t>0,\\
\sigma\rho_{x}+ h(t)b\rho=0,& x=0,1; t>0,\\
\rho\geq 0,\ \int\limits_0^1 \rho(x,t) dx =1,& t>0. 
\end{array}\right.
\ee

Here $b$ is the $x$-derivative of the potential $\psi$ and $h$ is given by \eqref{h}. The ratchet phase time periods are of length $T_{tr}$, and the pure diffusion periods are of length $T_{diff}=T-T_{tr}$. We denote $$\tau=\sigma T_{diff}.$$

Following \cite{kind1, heath}, the potential $\psi(x)$ is assumed to be a $C^4$-smooth function on $[0,1]$ of period $1/k$, with $k>1$ being a fixed integer, having maxima at points $x_i$ and minima at points $a_i$ and being monotone between these points (a ratchet-like form), where \be x_i=\frac {i-1} k,\ i=1,\dots, k+1,\ee \be a_i= a+x_i,\ i=1,\dots, k.\ee The positive parameter $a$ should be less than $1/k$. 

The problem can be completed with the initial condition
\be  \label{in}
\rho(x,0) = \rho_0(x),\quad x\in (0,1),\ee such that \be
\rho_0(x)\geq 0, \int\limits_0^1 \rho_0(x) dx =1.
\ee

The existence of a periodic orbit for \eqref{eq} is provided by
\begin{theorem} \label{exis}   (see \cite[Theorem 1]{kind1}) Assume that  \be2\pi^2 \tau -\lambda T_{tr}> \ln 2, \ee where $\lambda$ is a certain constant depending only on the potential (see \cite{kind1} for its exact value).  Then problem \eqref{eq} has a unique $T$-periodic solution $\rho^s$.  \end{theorem} 

It is also known \cite{kind1,kind2} that the periodic orbit $\rho^s$ eventually attracts all the solutions $\rho$ of  \eqref{eq}, namely, \be\lim_{n\to\infty,\,t_n=t+nT}\left\|\rho(\cdot,t_n)-\rho^s(\cdot,t_n)\right\|_{H^2}=0.\ee For brevity, sometimes we will write simply $\rho^s$ or $\rho^s(x)$ for the function $\rho^s(x,0)=\rho^s(x,nT)$. In \cite{kind1} it is shown that \be\|\rho^s\|_{H^2}\leq R_0=\frac{\sqrt 2 +1} {\sqrt 2 -1}.\ee

Denote  $$\hat\rho^s_i=\int\limits_{x_i}^{x_{i+1}}\rho^s(x)\,dx,\ i=1,\dots, k.$$

The main result of the paper is

\begin{theorem} \label{mainthm}
If $a<\frac 1 {2k}$, then there exist $\sigma$, $T_{tr}$ and $T$ such that \be\label{main}\hat\rho^s_1>\hat\rho^s_2>\dots>\hat\rho^s_k.\ee
\end{theorem}

Theorem \ref{mainthm} means, in particular, that, given any initial distribution of density, after a large number of cycles there will be more mass on the left than on the right.

\begin{remark} If  $a>\frac 1 {2k}$, then 
Theorem \ref{mainthm} implies $\hat\rho^s_1<\hat\rho^s_2<\dots<\hat\rho^s_k$, to see this it suffices to make the change of variables $x\to 1-x$.  \end{remark}

The \emph{discrete ratchet} acts a follows. During the ratchet phase it simply concentrates all the matter from any segment $[x_i,x_{i+1}]$ at the point $a_i$. Thus, if $$\mu_i^*=\int\limits_{x_i}^{x_{i+1}}\rho_0(x)\,dx,$$ then at the moment $T_{tr}$ the density becomes $$\sum\limits_{i=1}^k\mu_i^*\delta_{a_i}.$$ During the diffusion phase we have the same diffusion as for the flashing ratchet. Then this process is repeated periodically. 

Denote by $d$ the Wasserstein metric (see e.g. \cite{kind3}) on the space of probability measures on $[0,1]$. For a continuous function $f$ and a probability measure $p$ on $[0,1]$, we use the bra-ket notation as follows: $$\langle p,f\rangle=\int\limits_0^1 f \, dp.$$
The convergence in Wasserstein metric implies *-weak convergence of probability measures, i.e. \be\label{prob}d(p_n,p)\to 0 \Rightarrow \langle p_n-p,f\rangle\to 0,\ f\in C [0,1].\ee 

At the end of the ratchet phase, it is possible to estimate the distance between the solution to \eqref{eq},\eqref{in} and the outcome of the discrete ratchet:

\begin{lemma} (see \cite[Corollary 3]{kind1}) Let $\rho$ be a solution to \eqref{eq},\eqref{in} with $\|\rho_0\|_{H^2}\leq R_0$.
Then for sufficiently large $T_{tr}$ (the lower bound on $T_{tr}$ depends on the potential only) one has \be\label{wass} d\left(\rho(\cdot,T_{tr}), \sum\limits_{i=1}^k\mu_i^*\delta_{a_i}\right)^2\leq R_0(1+c_1)\frac{\ln^2 T_{tr}}{T_{tr}^2}+\min\{C_\lambda\sigma e^{\lambda T_{tr}/2},1\}.\ee The constants $c_1$ and $C_\lambda$ depend on the potential only.\end{lemma}

Let us describe how the discrete ratchet generates a Markov chain. Consider the heat equation with Neumann boundary conditions:
\be \label{heat}
\left\{\begin{array}{ll} y_s=y_{xx},
& x\in (0,1); s>0,\\
y_{x}=0,& x=0,1; s>0. 
\end{array}\right.
\ee

Let $$ \Gamma_{s}(x)=\frac{\exp(-x^2/4s)}{2\sqrt{\pi s}},$$ and
\be G(x,s)=\sum\limits_{n=-\infty}^{\infty} \Gamma_s(x+2n).\ee
Note that $G$ is $2$-periodic and even in $x$, and $G(1+x,s)=G(1-x,s)$.

The Green function for \eqref{heat} is \cite{pol}
\be \label{gr} \begin{array}{ll} g(\xi,x,s)=1+2\sum\limits_{n=1}^{\infty} \cos(n\pi \xi)\cos(n\pi x)\exp(-n^2\pi^2 s)\\=G(x+\xi,s)+G(x-\xi,s).\end{array}\ee

Now introduce the following matrix:
\be P=P(\tau)=(p_{ij}),\quad p_{ij}=\int\limits_{x_j}^{x_{j+1}} g(a_i,x,\tau)\,dx.\ee

Since the initial distribution of mass between the segments  $[x_i,x_{i+1}]$ is given by the vector $\mu^*=(\mu_i^*)$, the outcome of the discrete ratchet
at the moment $T$ will have the distribution of mass between the segments described by the vector $\mu^* P(\tau)$ (cf. \cite{kind1}), at $2T$ it will be $\mu^* P^2$, and so on. 

\section{The stationary vector of the Markov chain}

In order to prove the main theorem we need first to study the stationary vector of the Markov chain generated by the matrix $P$.

We recall that an $m \times m$-matrix with positive entries is called \emph{ergodic} if the sum of the elements in every row is equal to one. The eigenvalue $1$ of any ergodic matrix $A$ is simple, and there exists a unique vector $\xi$ satisfying \be\xi_i \geq 0,\ i=1,\dots,m,\ \sum\limits_{i=1}^m \xi_i=1, \xi=\xi P.\ee We will call it the \emph{stationary vector} of $A$. Let us also introduce the number $$\kappa(A)=\min_{ \sum\limits_{i=1}^m y_i=0, |y|=1}|yA-y|.$$ Note that, for an ergodic matrix $A$, this number is positive, and $\kappa(A)\to 1$ as all the elements of $A$ approach $1/m$. 

The matrix $P$ is ergodic. 
Denote by $\mu^s=(\mu^s_i)$ its stationary vector, which is also the stationary vector of the corresponding Markov chain. The following result holds.

\begin{theorem} \label{mark}
For $a<\frac 1 {2k}$ and $\tau$ large enough, there is a constant $c>0$ independent of $\tau$ such that \be\label{main1}\mu^s_1\geq\mu^s_2+c e^{-\pi^2\tau},\ \mu^s_2\geq\mu^s_3+ c e^{-\pi^2\tau},\dots,\mu^s_{k-1}\geq\mu^s_k+c e^{-\pi^2\tau}.\ee \end{theorem}


Its proof requires
\begin{lemma} \label{crit} Let $A=(a_{ij})$ be an ergodic $m \times m$-matrix satisfying the following conditions: a) for any column (say, $j$-th, $j<m$) there exists a number $s=s(j)$ so that one has $a_{ij}\leq a_{i,j+1}$ provided $i> s$, and $a_{ij} \geq 
a_{i,j+1}$ provided $i< s$, b) there exists a constant $d>0$ such that $$ \bar A_1\geq \bar A_2+d,\ \bar A_2\geq \bar A_3+ d,\dots,\bar A_{m-1}\geq\bar A_m+d,$$ where $\bar A_j$ is the sum of elements in the $j$-th column. Then the stationary vector $\xi$ of $A$ satisfies
\be \xi_1\geq \xi_2+Md,\ \xi_2\geq \xi_3+ Md,\dots, \xi_{m-1}\geq \xi_m+Md,\ee where $M$ is the minimum of the elements of the last (i.e. $m$-th) column.

\end{lemma}

\begin{proof} Consider the set \be \notag B=\left\{y\in\R^m\,\middle|\, \sum\limits_{i=1}^m y_i=1, y_1\geq y_2+Md,y_2\geq y_3+ Md,\right.\ee \be \notag
 \left.\phantom{\sum\limits_{i=1}^m}\dots, y_{m-1}\geq y_m+Md,y_m\geq M\right\}. \ee 
This set is compact and convex. Moreover, $B$ is invariant for the map $\mathcal{A}:y\mapsto yA$. In fact, let $y\in B$. Then  $$\sum\limits_{i=1}^m \sum\limits_{j=1}^m y_j a_{ji}=\sum\limits_{j=1}^m  y_j \sum\limits_{i=1}^m a_{ji}=1.$$ Fix any $l=1,\dots,m-1$. Then $$(yA)_l-(yA)_{l+1}=\sum\limits_{i=1}^m y_i (a_{il}-a_{i,l+1})\geq \sum\limits_{i=1}^m y_{s(l)} (a_{il}-a_{i,l+1})$$ $$= y_{s(l)} (\bar A_l-\bar A_{l+1})\geq y_{s(l)} d\geq Md.$$  
Finally, $$\sum\limits_{i=1}^m y_i a_{im}\geq M\sum\limits_{i=1}^m y_i=M.$$

By Brouwer's fixed point theorem, $\mathcal{A}$ has a fixed point in $B$, which should coincide with the stationary vector. 
\end{proof}

\begin{proof} (Theorem \ref{mark}) Let us show that $P$ satisfies the conditions of Lemma \ref{crit}. Set $s=\left[\frac {k+1} 2\right]$ (the integer part) for any $j$. Firstly, let us check if $p_{ij} \geq 
p_{i,j+1}$ for $i< s$. Consider the function $$\phi(y)=\int\limits_{y}^{y+\frac 1 k} g(a_i,x,\tau)\,dx,\ 0\leq y\leq \frac {k-1}k.$$ It suffices to show that it is decreasing.
Observe that $a_i<\frac 1 2$ and $\cos (\pi a_i)>0$ since $i<s$. Now, $$\phi^\prime(y)=g(a_i,y+1/k,\tau)-g(a_i,y,\tau)$$ $$=2\sum\limits_{n=1}^{\infty} \cos(n\pi a_i)(\cos(n\pi y+{n\pi}/ k)-\cos(n\pi y))\exp(-n^2\pi^2 \tau)$$
$$=2\cos(\pi a_i)(\cos(\pi y+\pi/ k)-\cos(\pi y))\exp(-\pi^2 \tau)+o(e^{-\pi^2\tau})$$
$$\leq2\cos(\pi a_i)(\cos(\pi/ k)-1)\exp(-\pi^2 \tau)+o(e^{-\pi^2\tau})\leq 0$$ for large $\tau$.

The claim that $p_{ij}\leq p_{i,j+1}$ for $i> s$ can be proven similarly taking into account that $a_i>\frac 1 2$ and $\cos (\pi a_i)<0$ for $i>s$.

Without loss of generality (i.e. for large $\tau$) we may assume that the function $G(x,\tau)$ is decreasing in $x$ on the segment $[a,1-a]$, and \be\label{este} G_x(x,\tau)\leq -Ce^{-\pi^2\tau}, a\leq x\leq 1-a,\ee with some constant $C>0$. Really, from \eqref{gr} we get the following representation:

\be G(x,\tau)=\frac 1 2 +\sum\limits_{n=1}^{\infty} \cos(n\pi x)\exp(-n^2\pi^2 \tau).\ee Thus,

$$ G_x(x,\tau)=-\sum\limits_{n=1}^{\infty} n\pi\sin(n\pi x)\exp(-n^2\pi^2 \tau)$$ $$=-\pi\sin(\pi x)\exp(-\pi^2 \tau)+o(e^{-\pi^2\tau})$$
$$\leq -\pi\sin(\pi a)\exp(-\pi^2 \tau)+o(e^{-\pi^2\tau})\leq -Ce^{-\pi^2\tau}.$$

Take any $l=1,\dots,m-1$. We have to see that \be\label{usl}  \bar P_l\geq \bar P_{l+1}+d\ee with some $d$ independent of $l$. Consider the function $$\varphi(y)=\sum\limits_{i=1}^k\int\limits_{y}^{y+\frac 1 k} g(a_i,x,\tau)\,dx,\ 0\leq y\leq \frac {k-1}k.$$ Then $$\varphi^\prime(y)=\sum\limits_{i=1}^k [g(a+(i-1)/k,y+1/k,\tau)-g(a+(i-1)/k,y,\tau)]$$ $$=\sum\limits_{i=1}^k [G(a+y+i/k,\tau)+G(a-y+(i-2)/k,\tau)$$ $$-G(a+y+(i-1)/k,\tau)-G(a-y+(i-1)/k,\tau)]$$ $$=G(a+y+1,\tau)+G(a-y-1/k,\tau)-G(a+y,\tau)-G(a-y+(k-1)/k,\tau)$$ $$=G(1-a-y,\tau)+G(y-a+1/k,\tau)-G(a+y,\tau)-G(a-y+(k-1)/k,\tau).$$

The length of the segments $[a+y,y-a+1/k]$ and $[a-y+(k-1)/k,1-a-y]$ is $\frac 1 k -2a$. Therefore \eqref{este} implies \be\varphi^\prime(y)\leq C(4a-2/ k)e^{-\pi^2\tau}.\ee Thus, \eqref{usl} holds with $$d=\frac {C(-4a+2/ k)e^{-\pi^2\tau}}{k}.$$ It remains to observe that due to \eqref{gr} one has \be M=\frac 1k+ O(e^{-\pi^2\tau})\ee  for the minimum of the last column of $P$.

\end{proof}

\section{Proof of the main theorem}

\begin{proof}(Theorem \ref{mainthm}) Observe that
 there exist sequences $$T_{tr,n}\to \infty, \tau_n\to\infty, \sigma_n\to 0$$ satisfying \be\sigma_n e^{\lambda T_{tr,n}/2}\to 0\ee and
  \be2\pi^2 \tau_n -\lambda T_{tr,n}> \ln 2.\ee Let $T_n=T_{tr,n}+\frac {\tau_n}{\sigma_n}$. Then we can find the corresponding $T_n$-periodic solutions $\rho^{s,n}$, and introduce obvious notations $\hat\rho^{s,n}$, $\mu^{s,n}$ etc. for the corresponding values. It suffices to show that \be|\hat\rho^{s,n}-\mu^{s,n}|= o (e^{-\pi^2\tau_n})\ee as $n\to \infty$. In this case, setting $\tau=\tau_n, T_{tr}=T_{tr,n}$ and $T=T_n$ for $n$ large enough, we would get \eqref{main} from \eqref{main1}.
  
 Let $$\rho^{*s,n}(x,t)= \sum\limits_{i=1}^k\hat\rho^{s,n}_i g(a_i,x,\sigma(t-T_{tr,n})),\quad t>T_{tr,n}.$$ Then $$
 |\hat\rho^{s,n}_j-(\hat\rho^{s,n}P(\tau_n))_j|$$ $$=\left|\int\limits_{x_j}^{x_{j+1}}\rho^{s,n}(x,T_n)-\rho^{*s,n}(x,T_n)\,dx\right|$$ $$=\left|\int\limits_{x_j}^{x_{j+1}}\left\langle\rho^{s,n}(\cdot,T_{tr})- \sum\limits_{i=1}^k\hat\rho_i^{s,n}\delta_{a_i}, g(\cdot,x,\tau_n)\right\rangle\,dx\right |
 $$ \be\label{ner}\begin{array}{c}=\left | \int\limits_{x_j}^{x_{j+1}}\left\langle\rho^{s,n}(\cdot,T_{tr})- \sum\limits_{i=1}^k \hat \rho_i^{s,n} \delta_{a_i} , 1 \right\rangle \right.
  \\ +\left. 2\left\langle\rho^{s,n}(\cdot,T_{tr})- \sum\limits_{i=1}^k\hat\rho_i^{s,n}\delta_{a_i},\cos(\pi
   \cdot)\right\rangle\cos(\pi x)e^{-\pi^2 \tau_n}\right.\\+\left. 2\sum\limits_{m=2}^{\infty} \left\langle\rho^{s,n}(\cdot,T_{tr})-
    \sum\limits_{i=1}^k\hat\rho_i^{s,n}\delta_{a_i}, \cos(m\pi \cdot)\right\rangle\cos(m\pi x)e^{-m^2\pi^2 \tau_n}\,dx\right
    |,\end{array}\ee $$\ j=1,\dots,n.$$ The first summand is zero, the third is $o (e^{-\pi^2\tau_n})$. Due to \eqref{prob} and \eqref{wass} with $\rho(x,t)=\rho^{s,n}(x,t)$, $\mu^*=\hat\rho^{s,n}$, one has
    $$\left\langle\rho^{s,n}(\cdot,T_{tr})- \sum\limits_{i=1}^k\hat\rho_i^{s,n}\delta_{a_i},\cos(\pi
   \cdot)\right\rangle\to 0,$$ so the second summand from \eqref{ner} is also $o (e^{-\pi^2\tau_n})$ as $n\to \infty$. Thus, \be
 |\hat\rho^{s,n}-\hat\rho^{s,n}P(\tau_n)|=o (e^{-\pi^2\tau_n}).\ee It remains to observe that \be|\hat\rho^{s,n}-\mu^{s,n}|\leq \frac{ |\hat\rho^{s,n}-\hat\rho^{s,n}P(\tau_n)|}{\kappa(P(\tau_n))},\ee and $\kappa(P(\tau_n))\to 1$.
 
 \end{proof}

\end{document}